\documentclass[reqno,11pt]{article}
\usepackage{amsmath,amssymb}             

\usepackage[utf8]{inputenc}
\usepackage[english]{babel}

\usepackage[style=alphabetic,doi=false,isbn=false,url=false]{biblatex}

\usepackage[left=1.3in,right=1in,top=0.8in,bottom=0.8in,includefoot,includehead,headheight=13.6pt]{geometry}
\usepackage{color}
\usepackage{hyperref}
\hypersetup{pdfborder={0 0 0}, colorlinks, citecolor=blue, urlcolor=blue}

\usepackage{amsthm}
\usepackage{mathtools}

\newtheorem{Theoreme}{Théorème}[section]

\newtheorem{Proposition}[Theoreme]{Proposition}

\newtheorem{Theorem}{Theorem}[section]

\newtheorem{Lemma}[Theorem]{Lemma}

\DeclareMathOperator{\PP}{\blackboard{P}}

\DeclareMathOperator{\RR}{\blackboard{R}}

\DeclareMathOperator{\W}{\mathbf{W}}

\DeclareMathOperator{\Hr}{\calig{H}}
\DeclareMathOperator{\BB}{B}
\DeclareMathOperator{\Lip}{Lip}

\newcommand{\iid}{\emph{i.i.d.}\,}
\newcommand{\ie}{\emph{i.e.}\,}

\newcommand{\calig}[1]{\mathcal{#1}}
\renewcommand{\H}{\calig{H}}

\newcommand{\blackboard}[1]{\mathbf{#1}} 
\newcommand{\N}{\blackboard{N}}
\newcommand{\EB}{\blackboard{E}}

\newcommand{\R}{\blackboard{R}}

\renewcommand{\d}[1]{\textup{d} #1}

\newcommand{\delimleft}[2]{\ifcase #1\or
    \bigl#2\or %
    \Bigl#2\or %
    \biggl#2\or %
    \Biggl#2\or %
    \left#2\fi}
\newcommand{\delimright}[2]{\ifcase #1\or
    \bigr#2\or %
    \Bigr#2\or %
    \biggr#2\or %
    \Biggr#2\or %
    \right#2\fi}
\newcommand{\pa}[2][5]{
    \delimleft{#1}{(} #2 \delimright{#1}{)}}
\newcommand{\br}[2][5]{
    \delimleft{#1}{[} #2 \delimright{#1}{]}}
\newcommand{\ac}[2][1]{
    \delimleft{#1}{\{} #2 \delimright{#1}{\}}}

\newcommand{\abs}[2][1]{
    {\delimleft{#1}{\lvert} #2%
    \delimright{#1}{\rvert}}}
\newcommand{\norm}[2][1]{
    {\delimleft{#1}{\lVert} #2%
    \delimright{#1}{\rVert}}}
\newcommand{\normLp}[3][1]{
    \norm[#1]{#3}_{\scriptscriptstyle #2}}

\newcommand{\un}{
    \boldsymbol{1}}

\newcommand{\esp}[2][5]{
    \EB \br[#1]{#2}}

\renewcommand{\le}{\leqslant}
\renewcommand{\ge}{\geqslant}
\newcommand{\ds}{\displaystyle}

\title{Weak error for nested Multilevel Monte Carlo}

\author{Daphn\'e Giorgi\footnote{Sorbonne Université, Sorbonne Paris Cité, CNRS, Laboratoire de Probabilités Statistique et Modélisation, LPSM, F-75005 Paris, France, E-mail: \url{daphne.giorgi@sorbonne-universite.fr}},
Vincent Lemaire\footnote{Sorbonne Université, Sorbonne Paris Cité, CNRS, Laboratoire de Probabilités Statistique et Modélisation, LPSM, F-75005 Paris, France,
E-mail: \url{vincent.lemaire@sorbonne-universite.fr}}, Gilles Pag\`es\footnote{Sorbonne Université, Sorbonne Paris Cité, CNRS, Laboratoire de Probabilités Statistique et Modélisation, LPSM, F-75005 Paris, France, E-mail: \url{gilles.pages@sorbonne-universite.fr}}
}

\begin{document}

\maketitle
  
\begin{abstract}
	This article discusses MLMC estimators with and without weights, applied to nested expectations of the form $\esp{f(\esp{F(Y,Z)|Y})}$. More precisely, we are interested on the assumptions needed to comply with the MLMC framework, depending on whether the payoff function $f$ is smooth or not. A new result to our knowledge is given when $f$ is not smooth in the development of the weak error at an order higher than 1, which is needed for a successful use of MLMC estimators with weights.
\end{abstract}

\noindent \emph{Keywords:} Multilevel Monte Carlo; Weighted Multilevel Monte Carlo; Nested Monte Carlo; Weak error expansion.

\noindent \emph{MSC 2010:} primary 65C05; secondary 65C30.

\section{Introduction}

Multilevel estimators are commonly used when the underlying random variable of interest~--~here~$f\pa{\esp{F(Y,Z)|Y}}$, with $Y$ and $Z$ independent as far as nested simulation is concerned~--~cannot be simulated exactly at a reasonable computational cost. However, such approximations ~--~here  $f\pa{\frac 1N \sum_{k=1}^N F(Y,Z_k)}$~-- induce some bias. Nested simulation is one of the two most popular setting where Multilevel method are implemented, the other being the numerical schemes associated to stochastic dynamics.

The optimal calibration and the resulting performances of Multilevel Monte Carlo estimators depend on the weak and strong error rate of convergence of these simulable proxies
toward  $f\pa{\esp{F(Y,Z)|Y}}$. By weak error, we mean here an expansion of the bias as a function of a given parameter $h$ representative of the (inverse) complexity.

The existence of weak error expansions at order one leads to the (regular and original) Multilevel Monte Carlo (MLMC) method  introduced by M. Giles in~\cite{Gi08}, whereas higher order expansions led naturally  to develop a weighted multilevel framework, called Richardson-Romberg Multilevel  method (ML2R) introduced in~\cite{LePa17}. However, the existence of such an expansion not only depends upon   random variable of interest and its approximations but also on the regularity of the ``payoff'' function $f$, as it has been widely popularized by the analysis of time discretization schemes of Brownian diffusion processes (see~\cite{TaTu90} and~\cite{BaTa196}).

The seminal result concerning the first order weak error expansion  for nested Monte Carlo simulation when $f$ is not regular~--~namely a quantile~--~is due to Gordy and Juneja in~\cite{GoJu10}.

For such indicator function the strong rate of convergence remains slow and reduces the efficiency of regular multilevel estimators since their performances are ruled by this strong convergence rate. In particular they no longer behave  as unbiased or almost unbiased estimators as it is the case for faster strong convergence regimes.

By contrast weighted multilevel estimators are still  almost unbiased in some sense but this performance strongly relies on higher order expansions of the weak error. So the main objective and result of this paper is to establish (see Proposition~\ref{prop:nested_weakII}) such a higher order expansion for non-smooth payoff function $f$ in a nested simulation  framework. 

Let us briefly recall the multilevel paradigm (see~\cite{PaNumProb}). Let $Y_0 \!\in L^2(\Omega, {\cal A}, \PP)$ be a random variable and $Y_h$, $h\!\in {\cal H}= \big\{\frac{\mathbf{h}}{n}, \, n\ge1\big\}$ be a family of approximations of $Y_0$ such that $\lim_{h\to 0}\|Y_h-Y_0\|_2 =0$ with a simulation cost of the form
\[
{\rm Cost}(Y_h) = \kappa h^{-1}
\]
so that the parameter $h$ is inverse linear in the complexity.  Its role in the weak expansion error will lead us to call it {\em bias parameter}.

The central idea behind the regular {MLMC} estimator is to consider a $R$-tuple of parameters $h_j = h/M^{j-1}$, $j=1,\ldots,R$ ($h\!\in {\cal H}$) and to write the telescopic sum
\[
\esp{Y_{h_R}} = \esp{Y_{h_1}}  + \sum_{j=2}^R \esp{Y_{h_j}-Y_{h_{j-1}}}
\]
which suggests to introduce the estimator (see~\cite{Gi08})
\begin{equation} \label{def::MLMC}
\widehat I^N_{h,R,q} = \frac{1}{N_1} \sum_{k=1}^{N_1} Y_{h}^{(1),k} + \sum_{j=2}^R \frac{1}{N_j} \sum_{k=1}^{N_j} \pa{Y_{h_j}^{(j),k} - Y_{h_{j-1}}^{(j),k}}
\end{equation}
where $\big( Y_{h_j}^{(j),k} \big)_{k=1,\ldots,N_j}$ are independent copies as $k$ varies of $Y^{(j)}_{h_j}$ itself ``attached''  to $Y^{(j)}_0$ where $(Y^{(1)}_0, \ldots,Y^{(R)}_0)$ are i.i.d. with the same distribution as $Y_0$. The size $N_j$ of each simulation at {\em level} $j$ is of the form $N_j = \lceil q_jN\rceil$, $j=1,\ldots,R$. 

If a first order weak expansion error assumption  
\begin{equation} \tag{$WE_{\alpha,1}$}
    \esp{Y_h} = \esp{Y_0} + c_1 h^{\alpha} +o(h^{\alpha})
\end{equation}
is fulfilled for some $\alpha>0$, then
\[
\esp{\widehat I^N_{h,R,q}} = \esp{Y_{h_R}}
 = \esp{Y_0} + c_1\frac{h}{M^{R-1}} + o\big(h/M^{R-1}\big).
 \]
 which dramatically reduces the bias compared to a crude Monte Carlo simulation based on i.i.d. copies of $Y_h$. At this stage the calibration of the allocation parameters $q_1,\ldots,q_{_R}$ across the $R$ levels relies on  a strong error convergence rate assumption
\begin{equation}
\label{strong_error} \tag{$SE_\beta$}
    \forall\, h,h'\!\in \Hr,\quad \normLp[1]{2}{Y_h-Y_{h'}} \le V_1 |h-h'|^{\beta},
\end{equation}
 or its variants (see~$e.g.$~\cite{PaNumProb} among other references where this conditions are discussed). Thus, it happens that $Y_h-Y_{h'}$ is replaced by a random variable $Y_{h,h'}$ satisfying~\eqref{strong_error} and such that $\esp{Y_{h,h'}} = \esp{Y_h-Y_{h'}}$. This calibration aims at minimizing the {\em effort} of the estimator $\widehat I^N_{h,q,R}$, that is the product of its variance by its complexity, given a prescribed Root Mean Square Error ({\emph RMSE}) level $\|\widehat I^N_{h,R,q}-Y_0\|_2 \le \varepsilon$.
 
If a higher order weak error expansion can be established, namely 
\begin{equation}
    \label{weak_error} \tag{$WE_{\alpha,R}$} 
\esp{Y_h} = \esp{Y_0} +\sum_{r=1}^R c_r h^{\alpha r} + o(h^{\alpha R}),
\end{equation}
then there exists {\em weights} $(\mathbf{w}_j)_{j=1,\ldots,R}$, {\em only depending on $\alpha$, $M$ and $R$} such that $\sum_{1\le j\le R} \mathbf{w}_j =1$ and satisfying
\[
 \sum_{j=1}^{R} \mathbf{w}_j \esp{Y_{h_j}}  = \esp{Y_0} +\widetilde{ \mathbf{w}}_{_{R+1}} c_R h^{\alpha R} + o\Big(h^{\alpha R}\Big).
\]
These weights, solution to a Vandermonde system (see~\cite{LePa17}), as well as $\widetilde{\mathbf{w}}_{_{R+1}}$ have closed formulas ($ \widetilde{\mathbf{w}}_{_{R+1}}= \sum_{i=1}^R \mathbf{w}_i n_i^{-\alpha R}$). This naturally leads to define the {\em weighted multilevel estimator} (or {\em Richardson-Romberg multilevel estimator}, ML2R) as 
\begin{equation}\label{def::ML2R}
\widetilde I^N_{h,R,q} = \frac{1}{N_1} \sum_{k=1}^{N_1} Y_{h}^{(1),k} + \sum_{j=2}^R \frac{\W_j^R}{N_j} \sum_{k=1}^{N_j} \pa{Y_{h_j}^{(j),k} - Y_{h_{j-1}}^{(j),k}},
\end{equation}
where $\mathbf{W}^R_j=  \mathbf{w}_j + \cdots+ \mathbf{w}_{_R}$, $j=1,\ldots,R$ and the $Y^{(j),k}_{h_j}$ are as above. One checks that such an estimator ``kills'' the bias  in a much more efficient manner yet since
\[
\esp{\widetilde I^N_{h,R,q}}  = \esp{Y_0} + \widetilde{ \mathbf{w}}_{_{R+1}} c_R h^{\alpha R} + o\Big(h^{\alpha R}\Big).
\]

Then $\widetilde I^N_{h,R,q} $ can be calibrated like the MLMC estimator to minimize its effort for prescribed \emph{RMSE}. For more precise results on the performances of these two families of estimators, we refer to~\cite{LePa17} or~\cite{Gi17} or~\cite{PaNumProb}. But the important fact to be kept in mind is that, as far as nested Monte Carlo simulations are concerned with $f = \mbox{\bf 1}_{[a,+\infty)}$ (see next section for the specification of the r.v. $Y_h$ for this purpose), the $\beta $ parameter is lower than $1$ (see Proposition~\ref{prop::strong_error_indicator}) so that, as a consequence, the ML2R estimator $\widetilde I^N_{h,R,q} $ behaves ``almost'' like an unbiased estimator, for which the cost is known to be $K \varepsilon^{-2}$, $K>0$ constant.  More precisely, if~\eqref{weak_error} holds for every {\em depth}  $R\ge 1$ and $\lim_{R\to\infty}|c_R|^{\frac 1R} = \widetilde c_\infty \in (0, +\infty)$, then
\[
\mathrm{Cost}\left(\widetilde I^{N(\varepsilon)}_{h(\varepsilon),R(\varepsilon),q(\varepsilon)} \right) \preceq  K_{\alpha, \beta,M}\varepsilon^{-2}\cdot e^{\frac{1-\beta}{\sqrt{\alpha}}\sqrt{2\log(1/\varepsilon)\log(M)} }, 
\]
where we recall that $f(\varepsilon) \preceq g(\varepsilon)$ if and only if $\limsup_{\varepsilon\to0} g(\varepsilon)/f(\varepsilon) \leq 1$, $K_{\alpha, \beta,M}>0$ is constant and we highlight that $e^{\frac{1-\beta}{\sqrt{\alpha}}\sqrt{2\log(1/\varepsilon)\log(M)} } = o\left( \varepsilon^{-\eta}\right)$ for all $\eta>0$. Note that some numerical experiments carried out in~\cite{LePa17} and in~\cite{Gi17} confirm the fact that weighted multilevel ML2R simulations outperform regular MLMC estimator.

The paper is organized as follows. In Section \ref{sec::nested} we give the description of the nested framework. In Section \ref{sec::useful_results} we give some useful results which will be valid in both frameworks, both $f$ smooth and not. Section \ref{sec::smooth} is devoted to the smooth case, with a particular attention to the antithetic approach, and in Section \ref{sec::non_smooth} we treat the non smooth case and we give a new result concerning the weak error.

\section{Nested Monte Carlo simulation}
\label{sec::nested}

The purpose of the so-called {\em nested} Monte Carlo method is to compute by simulation nested expectations of the form
\begin{equation*}
    \esp{ f\pa[1]{\esp{\Xi|Y}} }, 
\end{equation*}
where $(\Xi,Y)$ is an $\R\times \R^d$-valued  couple of random variables defined on a probability space $(\Omega,{\cal A}, \PP)$ satisfying $\Xi\!\in L^2$ and $f: \R \to \R$ is a specified function such that $f(\esp{\Xi|Y})\!\in L^2$. 

We assume that there exist a Borel function  $F:\R^d\times \R^q \to \R$ and a random vector $Z: (\Omega,{\cal A})\to \R^{q}$ independent of $Y$ such that
\begin{equation*}
	\Xi = F(Y, Z).
\end{equation*}
Let us introduce the Borel function $\phi_0:\R^d \to \R$  defined by $\phi_0(y) = \esp{F(y, Z)}$ so that one may set $\esp{F(Y, Z)|Y} = \phi_0(Y)$.
Then one has the following representation
\begin{equation*}
    \esp{\Xi | Y} = \phi_0(Y) = \int_{\R^{q}} F(Y, z) \PP_{Z}(\d z).
\end{equation*}
To comply with the multilevel framework, we set $K_0\in\N^*$ and  $\Hr = \ac{1/K, \, K\!\in K_0\N^*}$, 
\begin{equation*}
	 X_0 := \esp{\Xi|Y}, \quad   
	 X_h := \frac 1K \sum_{k=1}^K F(Y, Z_k) \;\text{ with }\; h = \frac{1}{K}\in\Hr, 
\end{equation*}
where $(Z_k)_{k\ge 1}$ is an i.i.d. sequence of random vectors with  the same  distribution as $Z$, defined on $(\Omega,{\cal A}, \PP)$  and independent of $Y$ (up to an enlargement of the probability space  if necessary) and
\begin{equation*}
	Y_0 := f(X_0), \quad   
	Y_h := f(X_h).
\end{equation*}
To prove that the nested Monte Carlo estimator satisfies the bias error expansion~\eqref{weak_error} and the strong approximation error~\eqref{strong_error}, we introduce the random functions, $\forall y \in \R^d,$ 
\begin{align}
    D(y) &= F(y, Z) - \esp{F(y, Z)}, \\
    E_h(y) &= \frac{1}{K} \sum_{k=1}^K \Big( F(y, Z_k) - \esp{F(y, Z)} \Big) = \frac{1}{K} \sum_{k=1}^K F(y, Z_k) - \phi_0(y).
\end{align}
Note that $E_h(y)$ is the statistical error of the inner Monte Carlo estimator, which can be rewrited as $\ds E_h(y) = \frac{1}{K} \sum_{k=1}^K D(y)^{(k)}$ where $(D(y)^{(k)})_{k \ge 1}$ is a sequence of \iid copies of $D(y)$, and that $E_h(Y) = X_h-X_0$. 

We distinguish between two main frameworks, depending on whether or not $f$ is smooth, a classical example of non-smoothness being $f= \un_{(a,b)}$ (see~\cite{DeLo09}).
When the function $f$ is smooth enough, say $f \in {\cal C}^{1+\rho}(\R,\R)$ with $\rho\!\in (0,1]$, a variant of the former Multilevel \emph{nested} estimator has been proposed in~\cite{BuHaRe15},~\cite{HaAl12} and~\cite{ChLi12} (see also~\cite{Gi15}) to improve the  rate of strong convergence in order to attain the asymptotically unbiased setting, namely~\eqref{strong_error} with $\beta>1$. 
A root $M \ge 2$ being given, the idea is  to replace in the successive refined levels of the MLMC and ML2R estimators (see~\eqref{def::MLMC} and~\eqref{def::ML2R}) the difference $Y_{\frac hM}-Y_h$ (where $h= \frac{1}{K}$, $K\!\in K_0\N^*$)  by an antithetic type as follows
\[
Y_{h,\frac hM} :=  f\left(\frac{1}{MK} \sum_{k=1}^{MK} F\big(Y, Z_{k}\big)\right)- \frac 1M \sum_{m=1}^M f\left(\frac 1K \sum_{k=1}^K F\big(Y, Z_{(m-1)K+k}\big)\right) .
\]
It is clear that $\esp{Y_{h,\frac hM} } = \esp{Y_{\frac hM}-Y_{h} }$.

Before getting into the smooth and non smooth case, we give some useful results that will be valid in both frameworks and will be used to establish the higher order of weak error expansion.

\section{Useful results}
\label{sec::useful_results}

Following Comtet~\cite{Co74}, we introduce the partial Bell polynomials $\BB_{n, k}$ for $n \ge 1$ and $k =1,\dots,n$ defined by 
\begin{equation} \label{def:partialBell}
	\BB_{n, k}(x_1, \dots, x_{n-k+1}) = \sum \frac{n !}{\ell_1! \cdots \ell_{n-k+1}!} \pa[2]{\frac{x_1}{1!}}^{\ell_1} \cdots\pa[2]{\frac{x_{n-k+1}}{({n-k+1})!}}^{\ell_{n-k+1}}  
\end{equation}
where the summation takes place over all integers $\ell_1, \dots, \ell_n \ge 0$, such that $\ell_1 + 2 \ell_2 + \cdots + (n-k+1) \ell_{n-k+1} = n$ and $\ell_1 + \cdots + \ell_{n-k+1} = k$. Note that $\deg (B_{n,k})=k$. 
The complete Bell polynomials $\BB_n$ are defined by 
\begin{equation*}
	\BB_n(x_1,\dots,x_n) = \sum_{k=1}^n \BB_{n, k}(x_1, \dots, x_{n-k+1}).
\end{equation*}
The first statement is a formal Taylor expansion with integral remainder of $\esp{g( X_h)}$ around $\esp{g( X_0)}$, with $g:\RR\to\RR$ a test function.

\begin{Lemma}[Taylor expansion] \label{lem:taylor} Let $R \ge 0$ and let $g:\R \to \mathbf{C}$ be a $2R+1$ times differentiable function. 
    
Assume $\Xi = F(Y, Z) \in L^{2R+1}$ and let $\kappa_j(\xi)$ be  the $j$--th cumulant (a.k.a. semi-invariant) of a random variable $\xi$. We set $\kappa_{j,y}:= \kappa_j(D(y))$ with $D(y) = F(y, Z)-\esp{F(y, Z)}$ for $y \in \R^d$ and $j \in \ac{1,\dots,R}$. Let $(\BB_{n,k})_{1\le k\le n}$ be  the {\em partial Bell polynomials} defined by~\eqref{def:partialBell}. We then define for $r \in \N$, $r+1 \le n \le 2r$ and every $y \in \R^d$, 
$$
b_{r,n-r}(y) = \BB_{r,n-r}\pa{\frac{\kappa_{2,y}}{2}, \dots, \frac{\kappa_{2r-n+2,y}}{2r-n+2}}.
$$
Then 
\begin{equation} \label{eq:taylor_nested}
    \forall h \in \Hr, \quad \esp{g( X_h)} = \esp{g( X_0)} 
+ \sum_{r=1}^{2R-1} c(r,(2r+1)\wedge 2R) h^r + \calig{R}_{2R+1},
\end{equation}
with 
\begin{equation} \label{eq:taylor_coeff}
 c(r,k) = \frac{1}{r!} \sum_{\ell = r+1}^{k} \esp{g^{(\ell)}( X_0) b_{r, \ell-r}(Y)}, \;  1 \le r < k \le 2r+1, 
\end{equation}
and 
\begin{equation} \label{eq:taylor_reste}
    \calig{R}_{2R+1} = \frac{1}{(2R)!} \esp{\int_0^{ X_h- X_0} g^{(2R+1)}(t+ X_0) ( X_h- X_0-t)^{2R} \d t}. 
\end{equation}
\end{Lemma}

\begin{proof}[Proof]
    The case $R=0$ is trivial, since it is a direct application of the fundamental theorem of calculus.

Let $R \ge 1$ be an integer. The Taylor formula at order $2R$ applied to $g$ at $\phi_0(y)$ reads
\begin{equation} \label{eq:taylorBell}
    \esp{g\pa[3]{\frac{1}{K} \sum_{k=1}^K F(y, Z_k)}} = g(\phi_0(y)) + \sum_{n = 1}^{2R} \frac{g^{(n)}(\phi_0(y))}{n !} \esp{(E_h(y))^n} + \calig{R}_{2R+1}(y),
\end{equation}
where $\calig{R}_{2R+1}(y) = \frac{1}{(2R)!} \esp{\int_{0}^{E_h(y)} g^{(2R+1)}(t + \phi_0(y)) (E_h(y) - t)^{2R} \d t}$.

The Bell polynomials allow us to explicitly compute the moments $\esp[1]{(E_h(y))^n}$, $n =1,\dots,R$ of $E_h(y)$ as follows. Let $\kappa_{j,y} = \kappa_j(D(y))$, $j = 1,\dots,R$, $y\in\R^d$. Additivity and homogeneity of cumulants give  
\begin{equation*}
    \forall j = 1, \dots, 2R-1, \quad \kappa_j\pa[1]{E_h(y)} = h^{j-1} \kappa_{j,y}.
\end{equation*}
Moments of $E_h(y)$ can be expressed in terms of cumulants using complete Bell polynomials (see~\cite{Co74} p.160 Equation$(2)$) as:
\begin{equation*}
	\esp[1]{(E_h(y))^n} = \BB_n\pa[2]{\kappa_1\pa[1]{E_h(y)}, \dots, \kappa_n\pa[1]{E_h(y)}}.
\end{equation*}
First note that $\kappa_{1,y} = 0$ so that $\kappa_1\pa[1]{E_h(y)} = 0$. Moreover, it follows from the definition~\eqref{def:partialBell} that $\BB_{n, k}$ is $k$--homogeneous, consequently
\begin{align*}
    \esp[1]{(E_h(y))^n} = h^{n} \sum_{k=1}^n h^{-k} \BB_{n, k}\pa[1]{0, \kappa_{2,y}, \dots, \kappa_{n-k+1,y}}. 
\end{align*}
We again derive from~\eqref{def:partialBell} that $\BB_{n,n} (0) = 0$, hence the last term in the above sum is null. In particular the sum in~\eqref{eq:taylorBell} starts from $n=2$.
Note now that  
\begin{equation*}
    \BB_{n, k}\pa{0, \kappa_{2,y}, \dots, \kappa_{n-k+1,y}}
    = \begin{cases}
        \frac{n!}{(n-k)!} b_{n-k, k}(y) & \text{if $1 \le k \le \lceil n / 2 \rceil$}, \\
        0 & \text{if $k > \lceil n / 2 \rceil$},
    \end{cases}
\end{equation*}
with $\ds b_{n-k, k}(y) = \BB_{n-k, k}\pa{\frac{\kappa_{2,y}}{2},\dots, \frac{\kappa_{n-2k+2,y}}{n-2k+2}}$ which implies that 
\begin{align} \label{eq:momentsBell}
    \esp[1]{(E_h(y))^n} = h^{n} \sum_{k=1}^{\lceil{n/2\rceil}} h^{-k} 
    \frac{n!}{(n-k)!} b_{n-k, k}(y).
\end{align} 
Plugging~\eqref{eq:momentsBell} in~\eqref{eq:taylorBell} gives, since the sum starts at $n=2$ as mentioned above, 
\begin{equation*}
    \esp{g\pa[3]{\frac{1}{K} \sum_{k=1}^K F(y, Z_k)}} 
    = g(\phi_0(y)) + \sum_{n = 2}^{2R} g^{(n)}(\phi_0(y))
    \sum_{k=1}^{\lceil n/2 \rceil} \frac{h^{n-k}}{(n-k)!} b_{n-k,k}(y)
	+ \calig{R}_{2R+1}(y).
\end{equation*}
Setting $r = n-k$ in the above expression, noting that $\lceil n/2\rceil + \lfloor n/2 \rfloor = n$ and that $\lfloor n/2 \rfloor\le r$ if and only if $n\leq 2r+1$, on derives by interchanging the sums that 
\begin{equation*}
\esp{g\pa[3]{\frac{1}{K} \sum_{k=1}^K F(y, Z_k)}} = g(\phi_0(y)) + \sum_{r = 1}^{2R-1} \frac{h^{r}}{r!} \pa{\sum_{n = r+1}^{(2r+1) \wedge 2R} g^{(n)}(\phi_0(y)) b_{r, n-r}(y)}
	+ \calig{R}_{2R+1}(y).
\end{equation*}
We conclude by integrating with respect to $\PP_Y(\d y)$.
\end{proof}

Taking advantage of this expansion we will derive two results. First a bias error expansion for smooth enough payoff functions, in which no regularity is required on the law of $( X_0,  X_h)$ (see Subsection~\ref{sec:smooth_bias_error}). Conversely a second result will be established relying on the regularity of the distribution of $(X_0,  X_h)$ when the payoff function is not smooth (see Subsection~\ref{sec:non_smooth_bias_error}). 

As concerns the strong error, elementary computations show that, if $\Xi\!\in L^2$, then 
\begin{equation}
\label{eq::nested_X_L2_bound}
    \normLp[1]{2}{X_h- X_0}^2 = \frac{1}{K} \int \PP_{Y}(\d y) {\rm var}\big(F(y,Z)\big) =  h \esp{(F(Y,Z)-\phi_0(Y))^2}  \leq  h\,{\rm var}\big(F(Y,Z)\big),
\end{equation}
since $\phi_0(Y) = \esp{F(Y,Z)|Y}$. To prove~\eqref{strong_error} we extend this result to $\normLp[1]{p}{X_h - X_{h'}}$ when $\Xi \in L^p$, $p > 1$, as described in Lemma~\ref{lemma::strong_conv_useful}. 

Note that from now on we give the results for a generic $p>1$ instead of $p=2$, because this wider assumption can be useful to establish a condition of uniform  integrability needed to prove a Central Limit Theorem (and strong law of large numbers) for Multilevel Monte Carlo estimators, see Lemma 5.2 in~\cite{GiLePa17}.

The proof of Lemma~\ref{lemma::strong_conv_useful} relies on the Marcinkiewicz-Zygmund inequality that we recall for clarity. If $(\xi_n)_{n \ge 1}$ is a sequence of centered independent random variables such that $\esp{\abs{\xi_n}^p} < +\infty$, $1 < p < +\infty$, then
\begin{equation} \label{eq:marcinkZ}
    \normLp[4]{p}{\sum_{k=1}^K \xi_k} \le (B_p)^{\frac{1}{p}} \normLp[4]{\frac{p}{2}}{\sum_{k=1}^K \xi_k^2}^{\frac{1}{2}}, 
\end{equation} 
where $B_{p} = \frac{18p^{\frac 32}}{(p-1)^{\frac 12}}$ (see~\cite{Sh96} p.499). If moreover $(\xi_n)_{n \ge 1}$ are identically distributed we have 
\begin{equation} \label{eq:marcinkZ_iid}
    \normLp[4]{p}{\sum_{k=1}^K \xi_k} \le (B_p)^{\frac{1}{p}} \sqrt{K} \normLp[1]{p}{\xi_1}
\end{equation}
We make an intensive use of this inequality in Section~\ref{sec::smooth}.

\begin{Lemma}\label{lemma::strong_conv_useful}
Assume $\Xi\!\in L^p$, $p > 1$.  Then, for every $h$, $h'\!\in \Hr$,
\begin{equation}\label{eq:Yh-Yhprime}
    \normLp[1]{p}{X_h- X_{h'}} \le 2B_p \normLp[1]{p}{\Xi - \esp{\Xi|Y}} |h-h'|^{\frac 12}.
\end{equation}
\end{Lemma}

\begin{proof}[Proof] Assume first that $h' \le h$, and set $K=\frac 1h$, $K'=\frac{1}{h'} \ge K$. 
First note that by Fubini's theorem
\[
    \normLp[1]{p}{X_h- X_{h'}}^p = \int_{\R^d} \PP_{_Y}(\d y) \esp{\left|E_h(y) - E_{h'}(y)
\right|^p}.
\]
Setting $\widetilde F(y,z) = F(y,z)-\phi_0(y)$, we write 
\begin{equation*}
    E_h(y) - E_{h'}(y) = \pa{\frac{1}{K} - \frac{1}{K'}} \sum_{k=1}^K \widetilde F(y,Z_k) + \frac{1}{K'} \sum_{k=K+1}^{K'} \widetilde F(y,Z_k).
\end{equation*}
Then, for every $y\!\in \R^d$, it follows from Minkowski's Inequality,
\begin{equation*}
    \normLp{p}{E_h(y) - E_{h'}(y)} 
        \le |h-h'| \normLp[4]{p}{\sum_{k=1}^K \widetilde F(y,Z_k)}
        + h'\normLp[4]{p}{\sum_{k=K+1}^{K'} \widetilde F(y,Z_k)}.
\end{equation*}
Applying Marcinkiewicz-Zygmund Inequality to both terms on the right hand side of the  above inequality yields
\begin{align*}
    \normLp{p}{E_h(y) - E_{h'}(y)} 
    &\le |h-h'| B_p \normLp[4]{\frac{p}{2}}{\sum_{k=1}^K \widetilde F(y,Z_k)^2}^{\frac 12}
    + h' B_p \normLp[4]{\frac{p}{2}}{\sum_{k=K+1}^{K'}\widetilde F(y,Z_k)^2}, \\
    &\le |h-h'|B_p K^{\frac 12} \normLp[1]{p}{\widetilde F(y,Z)}
    + h'B_p(K'-K)^{\frac 12} \normLp[1]{p}{\widetilde F(y,Z)}.
\end{align*}
Finally, for every $y\!\in \R^d$, 
\begin{align*}
    \normLp{p}{E_h(y) - E_{h'}(y)} 
    &\le B_p \normLp[1]{p}{\widetilde F(y,Z)} \left( (h-h') \frac{1}{\sqrt{h}} +h'\Big(\frac{1}{h'}-\frac{1}{h}\Big)^{\frac 12}\right)\\
    &= B_p \normLp[1]{p}{\widetilde F(y,Z)} (h-h')^{\frac 12}\left( \Big(1-\frac{h'}{h}\Big)^{\frac12}  +\Big( \frac{h'}{h}\Big)^{\frac 12}\right)\\
    &\le 2 B_p \normLp[1]{p}{\widetilde F(y,Z)} (h-h')^{\frac 12}.
\end{align*}
Plugging this bound in the above equality yields, owing to Minkowski's Inequality and Jensen's Inequality for conditional expectations, the announced result
\begin{align*}
    \normLp{p}{X_h - X_{h'}}^p & \le (2 B_p)^p  \int_{\R^d} \PP_{_Y}(\d y) 
    \normLp[1]{p}{\widetilde F(y,Z)}^p (h-h')^{\frac p2}\\
    &= (2 B_p)^p \normLp[1]{p}{\Xi-\esp{\Xi|Y}}^p (h-h')^{\frac p2}.
\end{align*}
\end{proof} 

\section{Smooth payoff function}
\label{sec::smooth}

We first focus on the smooth case, where we give a bias error expansion and a strong convergence rate when the payoff function $f$ is smooth. This result beyond its direct application will be an important step when dealing with indicator functions.

\subsection{Weak error}
\label{sec:smooth_bias_error}
The bias error expansion of the nested Monte Carlo estimator when $f$ is smooth is a consequence of Lemma~\ref{lem:taylor}, as we emphasized in the proof of the following Proposition.

\begin{Proposition} [Bias error (I): smooth functions] \label{prop:nested_weakI} Let $R \in \N^*$ and let $f:\R \to \R$ be a $2R+1$ times differentiable payoff function with bounded derivatives $f^{(k)}$, $k=R+1,\dots,2R+1$.  Assume $X \in L^{2R+1}$. Then there exists $c_1, \dots, c_R$ such that  
\begin{equation} 
    \forall h \in \Hr, \quad \esp{f( X_h)} = \esp{f( X_0)} + \sum_{r=1}^{R} c_r h^r + \calig{O}(h^{R+1/2}).
\end{equation}
\end{Proposition}
\begin{proof}[Proof]
    Applying Lemma~\ref{lem:taylor} with the function $g=f$ we get, for every $h \in \Hr$,   
\begin{equation}
    \esp{f( X_h)} = \esp{f( X_0)} + \sum_{r=1}^{R-1} c(r,2r+1) h^r + c(R,2R) h^R + \sum_{r=R+1}^{2R-1} c(r,2R) h^r + \calig{R}_{2R+1},
\end{equation}
with $c(r,k)$ defined in~\eqref{eq:taylor_coeff} and $\calig{R}_{2R+1}$ in~\eqref{eq:taylor_reste}. Establishing  the proposition amounts  to proving that the remainder term $\calig{R}_{2R+1}$ is well controlled. Using that $f^{(2R+1)}$ is bounded, we have 
\begin{equation*}
	\abs{\calig{R}_{2R+1}} \le \frac{\norm{f^{(2R+1)}}_{\infty}}{(2R+1)!} \esp{\abs{ X_h -  X_0}^{2R+1}}.
\end{equation*}
Using successively  the Marcinkiewicz-Zygmund Inequality and the Minkowski Inequality for the  $L^{R+\frac 12}(\PP)$-norm, we get, keeping in mind that $h= \frac 1K$, 
\begin{align*}
    \esp{\abs{E_h(y)}^{2R+1}} 
    & \le (B_{2R+1})^{2R+1} h^{2R+1} \esp{\abs[3]{
\sum_{k=1}^K (F(y, Z_k) - \phi_0(y))^2}^{R + 1/2}} \\
& \le (B_{2R+1})^{2R+1} h^{R+1/2} \esp{\abs{F(y, Z) - \phi_0(y)}^{2R+1}}
\end{align*}
Integrating with respect to $\PP_{Y}$ finally yields
\begin{align}
\nonumber    \esp{\abs{X_h - X_0}^{2R+1}} 
    &\le (B_{2R+1})^{2R+1} h^{R+1/2} \esp{\abs{F(Y, Z) - \phi_0(Y)}^{2R+1}}\\
\nonumber    &\le 2^{R+\frac 12} (B_{2R+1})^{2R+1} h^{R+1/2}\esp{|\Xi|^{2R+1} + |\esp{\Xi|Y}|^{2R+1}} \\
    &\le 2^{R+\frac 32} (B_{2R+1})^{2R+1} h^{R+1/2}\esp{|\Xi|^{2R+1}}\label{eq:Yh-Y_0:2R+1}
\end{align}
so that $\abs{\calig{R}_{2R+1}} = \calig{O}(h^{R+1/2})$.
\end{proof}

\subsection{Strong convergence rate}
If we assume that $f$ is Lipschitz continuous, Lemma~\ref{lemma::strong_conv_useful} straightforwardly shows that the standard \emph{nested} Monte Carlo satisfies a strong convergence at a rate $h^\beta$ with $\beta=1$. More precisely, if $\Xi \in L^2$, we have 
\begin{equation*}
    \normLp[1]{2}{Y_h - Y_{h'}} \le 2 B_2 [f]_{\Lip} \normLp[1]{2}{\Xi - \esp{\Xi|Y}} \abs{h - h'}^{\frac{1}{2}},
\end{equation*}
where $[f]_{\Lip}$ denotes the Lipschitz coefficient of $f$.

When asking for more smoothness, more precisely that $f'$ is $\rho$--Hölder, we can build an antithetic version of the \emph{nested} Monte Carlo which attains a strong convergence at a rate $h^\beta$ with $\beta>1$. As we saw, this corresponds to the optimal unbiased setting in terms of minimization of the computational cost. This antithetic multilevel estimator is obtained by replacing each difference $Y_{\frac hM}-Y_h$, $h \in h_1,\dots,h_{R-1}$, in the MLMC~\eqref{def::MLMC} and ML2R~\eqref{def::ML2R} estimators by the following random variable
\[
    Y_{h,\frac hM} :=  f\left(\frac{1}{MK} \sum_{k=1}^{MK} F\big(Y, Z_{k}\big)\right)- \frac 1M \sum_{m=1}^M f\left(\frac 1K \sum_{k=1}^K F\big(Y, Z_{(m-1)K+k}\big)\right),
\]
satisfying $\esp{Y_{h,\frac hM} } = \esp{Y_{\frac hM}-Y_{h} }$.
We set 
\begin{equation*}
\bar X_{K,m} = \frac 1K \sum_{k=1}^K F(Y, Z_{K(m-1)+k}) \quad \mbox{and} \quad \bar X_{MK} = \frac 1M \sum_{m=1}^M \bar X_{K,m} = \frac 1{MK} \sum_{k=1}^{MK} F(Y,Z_k),
\end{equation*}
so that the nested antithetic MLMC estimator~\eqref{def::MLMC} then reads
\begin{equation*}
    \widehat I^N_{h,R,q} = \frac{1}{N_1} \sum_{i=1}^{N_1} f(\bar X_{K,1}^{(i)}) + \sum_{j=2}^R \frac 1{N_j} \sum_{i=1}^{N_j} \left( f(\bar X_{MK}^{(i)}) - \frac 1M \sum_{m=1}^M f(\bar X_{K,m}^{(i)})\right),
\end{equation*}
with $(\bar X_{K,m}^{(i)})_{i\geq 1}$ independent copies of $\bar X_{K,m}$, and similarly for the ML2R estimator~\eqref{def::ML2R} with the weights $(W_j)_{2 \le j \le R}$.

\begin{Proposition}
Let $p > 1$ and $0 < \rho \le 1$. Assume $\Xi \in L^{p(1+\rho)}$ and $f'$ $\rho$--Hölder, i.e.
\begin{equation} \label{hp::rho_holder}
    \forall x, y \in \R, \quad \left| f'(x) - f'(y) \right| \leq \left[ f'\right]_{\rho} |x-y|^\rho.
\end{equation}
Then $Y_{h, \frac hM}$ satisfies a strong approximation error control similar as~\eqref{strong_error} with $\beta = 1+\rho>1$. More precisely, we prove that there exists $\widetilde V_1 > 0$ depending only on $p, \rho, [f']_\rho$ and $M$ such that 
\begin{equation} \label{eq:prop_strong_antithetic}
    \normLp[2]{p}{Y_{h, \frac{h}{M}}} \le \widetilde V_1 \pa{h - \frac{h}{M}}^{\frac{1 + \rho}{2}}.
\end{equation} 
\end{Proposition}

\begin{proof}[Proof]
Owing to Taylor's formula, for all $m=1, \ldots, M$, there exists $x_m$ in the geometric segment $\left(\bar X_{K,m}, \bar X_{MK} \right)$ such that
$$f(\bar X_{K,m}) = f(\bar X_{MK}) + f'(\bar X_{MK}) (\bar X_{K,m} -\bar X_{MK}) + \left( f'(x_m)-f'(\bar X_{MK}) \right) (\bar X_{K,m}-\bar X_{MK}).$$
Hence, using the definition of $\bar X_{MK}$,
\begin{equation} \label{eq::f_a_i_sum}
 \frac 1M \sum_{m=1}^M f(\bar X_{K,m}) = f(\bar X_{MK}) + \frac 1M \sum_{m=1}^M \left( f'(x_m) - f'(\bar X_{MK})\right)(\bar X_{K,m}-\bar X_{MK}).
\end{equation}
We aim at computing $\ds\normLp[2]{p}{Y_{h,\frac hM}} = \normLp[4]{p}{\frac 1M \sum_{m=1}^M f(\bar X_{K,m}) - f(\bar X_{MK})}$. Owing to the decomposition~\eqref{eq::f_a_i_sum}, to Minkowski's Inequality and to the $\rho$--Hölder assumption~\eqref{hp::rho_holder} on $f'$, we get
\begin{align} \label{eq::rho_inequality}
    \normLp[4]{p}{\frac 1M \sum_{m=1}^M f(\bar X_{K,m}) - f(\bar X_{MK})}
    &= \normLp[4]{p}{\frac 1M \sum_{m=1}^M \left( f'(x_m) - f'(\bar X_{MK})\right)(\bar X_{K,m}-\bar X_{MK})} \nonumber\\
    &\leq [f']_\rho \frac 1M \sum_{m=1}^M \normLp[2]{p}{X_{K,m}-\bar X_{MK}|^{1+\rho}}.
\end{align}
We first notice by an exchangeability argument that the variables $(\bar X_{K,m}-\bar X_{MK})_{m=1\ldots,M}$ are identically distributed with $\bar X_{K,m}-\bar X_{MK} \sim \bar X_{K,1}-\bar X_{MK}$. Moreover we write 
$$\bar X_{K,1} - \bar X_{MK} = \bar X_{K,1} - \frac 1M \sum_{m=1}^M \bar X_{K,m} = \frac 1M \sum_{m=1}^M (\bar X_{K,1} -\bar X_{K,m}) = \frac 1M \sum_{m=2}^M (\bar X_{K,1} -\bar X_{K,m}).$$
Hence, we get
\begin{equation} \label{eq::main-inequality}
    \normLp[4]{p}{\frac 1M \sum_{m=1}^M f(\bar X_{K,m}) - f(\bar X_{MK})}
    \leq [f']_\rho \frac 1{M^{1+\rho}} \normLp[4]{p}{\left|\sum_{m=2}^M (\bar X_{K,1} -\bar X_{K,m})\right|^{1+\rho}}.
\end{equation}
Owing to the independence of $Y$ and $(Z_k)_{k\geq1}$, we may write
\begin{align*}
 &\esp{\left|\sum_{m=2}^M (\bar X_{K,1} -\bar X_{K,m})\right|^{(1+\rho)p}} \\
 &= \int \PP_Y(\d y) \esp{\left|\sum_{m=2}^M \left(\frac 1K \sum_{k=1}^K F(y, Z_{k}) -\frac 1K \sum_{k=1}^K F(y, Z_{K(m-1) + k})\right)\right|^{(1+\rho)p}} \\
 &=  \int \PP_Y(\d y) \frac{1}{K^{(1+\rho)p}} \esp{\left| \sum_{k=1}^K  \left( (M-1)F(y, Z_{k}) - \sum_{m=2}^M  F(y, Z_{K(m-1) + k})\right)\right|^{(1+\rho)p}}.
\end{align*}
We notice that, for each fixed $y\in\RR^d$, the random variables $\xi_k = (M-1)F(Z_{k},y) - \sum_{m=2}^M  F(y, Z_{K(m-1) + k})$, $k \ge 1$, are centered and \iid. Moreover $(1+\rho)p > 1$ hence, owing to Marcinkiewicz-Zygmund inequality~\eqref{eq:marcinkZ_iid}, we have
\begin{multline*}
    \esp{\left|\sum_{m=2}^M (\bar X_{K,1} -\bar X_{K,m})\right|^{(1+\rho)p}} \\
    \leq B_{(1+\rho)p} \frac{1}{K^{\frac{(1+\rho)p}{2}}} \int \PP_Y(\d y) \esp{\left|  (M-1)F(y, Z_{1}) - \sum_{m=2}^M  F(y, Z_{K(m-1) + 1})\right|^{(1+\rho)p}}.
\end{multline*}
Applying twice Minkowski's Inequality yields
\begin{equation*}
    \esp{\left|\sum_{m=2}^M (\bar X_{K,1} -\bar X_{K,m})\right|^{(1+\rho)p}} 
    \leq C_{p,\rho} \frac{1}{K^{\frac{(1+\rho)p}{2}}} (M-1)^{(1+\rho)p} 2^{(1+\rho)p} \esp{\left|  F(Y, Z_{1}) \right|^{(1+\rho)p}}.
\end{equation*}
Plugging this in~\eqref{eq::main-inequality} we get
\begin{equation*}
    \normLp[4]{p}{\frac 1M \sum_{m=1}^M f(\bar X_{K,m}) - f(\bar{x})}^p \leq  \widetilde V_1 \pa{1 - \frac{1}{M}}^{\frac{p(1+\rho)}{2}} \frac 1{K^{\frac{p(1+\rho)}{2}}}, 
\end{equation*}
with $\ds \widetilde V_1 = [f']_\rho^p C_{p,\rho} 2^{p(1+\rho)} \left( 1-\frac1M\right)^{\frac{(1+\rho)p}{2}} \normLp{(1+\rho) p}{\Xi}^{(1+\rho) p}$, and~\eqref{eq:prop_strong_antithetic} is proved.
\end{proof}

If we replace the $\rho$--Hölder assumption on $f'$ by a weaker assumption $f'$ locally $\rho$--Hölder, $i.e.$
$$\forall x,y\in\R, \quad \left| f'(x) - f'(y) \right| \leq C |x-y|^\rho \left( 1 + |x|^q + |y|^q \right),$$
a strong convergence assumption with $\beta = 1 + \rho > 1$ similar to~\eqref{eq:prop_strong_antithetic} can still be  proved. Since $|x_m|^q\leq\max(|\bar X_{K,m}|^q,|\bar X_{MK}|^q)\leq |\bar X_{K,m}|^q+|\bar X_{MK}|^q$, Inequality~\eqref{eq::rho_inequality} must be replaced by
\begin{multline} \label{eq::locally_rho_inequality}
 \normLp[4]{p}{\frac 1M \sum_{m=1}^M f(\bar X_{K,m}) - f(\bar X_{MK})} = \normLp[4]{p}{\frac 1M \sum_{m=1}^M \left( f'(x_m) - f'(\bar X_{MK})\right)(\bar X_{K,m}-\bar X_{MK})} \\
  \leq [f']_\rho \frac 1M \sum_{m=1}^M \normLp[2]{p}{|\bar X_{K,m}-\bar X_{MK}|^{1+\rho}(1+|\bar X_{K,m}|^q+2|\bar X_{MK}|^q)}.
\end{multline}
Owing to Hölder's Inequality with $r,s>1$ such that $\frac1r +\frac 1s = 1$ and  Minkowski's Inequality, we get 
\begin{multline*} 
    \normLp[2]{p}{|\bar X_{K,m}-\bar X_{MK}|^{1+\rho}(1+|\bar X_{K,m}|^q+2|\bar X_{MK}|^q)}
    \\ \leq  \normLp[2]{pr}{|\bar X_{K,m}-\bar X_{MK}|^{1+\rho}}
    \left( 1+ \normLp[2]{ps}{|\bar X_{K,m}|^q} + 2 \normLp[2]{ps}{|\bar X_{MK}|^q} \right).
\end{multline*}
Since the variables $(\bar X_{K,m})_{m=1,\ldots,M}$ are identically distributed, Inequality ~\eqref{eq::locally_rho_inequality} yields
\begin{multline*}
    \normLp[4]{p}{\frac 1M \sum_{m=1}^M f(\bar X_{K,m}) - f(\bar X_{MK})} \\
    \leq [f']_\rho \normLp[2]{pr}{|\bar X_{K,1}-\bar X_{MK}|^{1+\rho}} 
    \left( 1+ \normLp[2]{ps}{|\bar X_{K,m}|^q} + 2 \normLp[2]{ps}{|\bar X_{MK}|^q} \right).
\end{multline*}

The analysis of the term $\normLp[2]{pr}{|\bar X_{K,1}-\bar X_{MK}|^{1+\rho}}$  does not change, except for the condition $\Xi\in L^{(1+\rho)pr}$. Under the assumption $\Xi \in L^{qps}$, the term $\normLp[2]{ps}{|\bar X_{K,m}|^q} + 2 \normLp[2]{ps}{|\bar X_{MK}|^q}$ is bounded, since
\begin{equation*}
    \normLp[2]{ps}{X_{K,1}|^q} = \normLp[4]{ps}{\left|\frac 1K \sum_{k=1}^K F(Y, Z_k) \right|^q} 
    \leq \normLp[1]{qps \vee 1}{\Xi}^q.
\end{equation*} 
Keeping in mind that $r = s/(s-1)$, the optimal choice for $s$ which minimizes both $(1+\rho)pr$ and $qps$ is given by $s = (1+\rho+q)/q$ (hence $r=(1+\rho+q)/(1+\rho)$). This leads to  the additional condition $\Xi\in L^{p(1+\rho+q)}$.
In conclusion, if $f'$ is locally $\rho$--Hölder, under the assumption $\Xi \in L^{p(1+\rho+q)}$, $Y_{h, \frac hM}$ satisfies the $L^p$ version of the strong convergence assumption with $\beta = 1+\rho>1$, similarly to~\eqref{eq:prop_strong_antithetic}.

\section{Indicator function and smooth density}
\label{sec::non_smooth}
There are many situations where we need to consider  non smooth payoff functions of the type $f=\un_{\{g(\esp{\Xi|Y})\in I\}}$, with $g:\R\to\R$ and $I\subset \R$ interval.
Among them we can cite the  computation of loss thresholds, \ie when we search, a threshold $q\in\RR$ being fixed, for the corresponding $\alpha_q\in[0,1]$ such that 
$$1-\alpha_q = \PP (g(\esp{\Xi|Y})\geq q) = \esp{\un_{\{ g(\esp{\Xi|Y})\geq q\}}},$$ 
or the inverse problem, which consists in computing the quantile $q_\alpha$ such that for a fixed $\alpha\in[0,1]$,
$$1-\alpha = \PP (g(\esp{\Xi|Y})\geq q_\alpha) = \esp{\un_{\{ g(\esp{\Xi|Y})\geq q_\alpha\}}}.$$
Another situation of interest is the approximation of density functions (see the seminal paper of Bally and Talay~\cite{BaTa196} and~\cite{BaTa296}, treating the law of the Euler scheme for distributions). 

The payoff function $f$ being non smooth, the regularity assumptions on $f$ that we needed to prove the weak and the strong convergence of the estimator in the smooth case, will be replaced by some regularity assumptions on the density functions, as we detail in the next two Subsections. 

\subsection{Weak error}
\label{sec:non_smooth_bias_error}
We recall the notation that $X_0 = \esp{\Xi|Y}$ and $X_h = \frac{1}{K} \sum_{k=1}^K F(Y, Z_k)$ with $h = \frac{1}{K} \in \Hr$ and we introduce the notation 
\begin{equation*}
    \Delta_h = X_h - X_0.
\end{equation*}

The following result on the weak error derives from Lemma~\ref{lem:taylor} and gives a bias error expansion relying on the density of the joint distribution of $(X_0,  \Delta_h)$ and of $(X_0,  Y)$. More precisely, assume that $(X_0,  \Delta_h)$ is a random vector with smooth density with respect to the Lebesgue measure on $\R^2$. Let $f_{X_0}$ be the density of $X_0$, let $f_{X_0, Y}$ be the  density of $(X_0, Y)$ and let $f_{X_0, \Delta_h}$ be the  density of $(X_0, \Delta_h)$.  
Moreover let $F_{ X_h}(x)$ and $F_{ X_0}(x)$ be the cumulative distribution functions of $ \Delta_h $ and $ X_0$.

\begin{Proposition} [Bias error (II): smooth density] \label{prop:nested_weakII}  $(a)$ Let $R \ge 0$. Assume that the partial derivatives $\partial_x^{(\ell)} f_{X_0,Y}(x,y)$ exist for $\ell=1,\ldots, 2R$, that the partial derivatives $\partial_x^{(\ell)} f_{X_0,\Delta_h}(x,y)$ exist for $\ell=1,\ldots, 2R+1$ and that $\partial_x^{(2R+1)} f_{X_0,\Delta_h}(x,y)$ is continuous.
    Assume that $\Xi \in L^{2R+1}$. 

    Let 
    \begin{equation*}
     P_{r}(x) =\frac{1}{ f_{ X_0}(x)} \sum_{\ell=r+1}^{(2r+1)\wedge 2R}(-1)^{\ell} \int_{\R} b_{r,\ell-r} (y) \partial_x^{(\ell)} f_{X_0,Y}(x,y) \d y.
    \end{equation*}
    
    \noindent $(a)$ If $\sup_{h\in \Hr, x, v \in \R}\left|\partial_x^{(2R+1)} f_{X_0 | \Delta_h = v}(x)\right|<+\infty$, then 
    \begin{equation}\label{eq:fY_0}
        f_{ X_h}(x) = f_{ X_0}(x) + f_{ X_0}(x) \sum_{r=1}^{R} \frac{h^r}{r!} P_{r}(x)  + \calig{O}(h^{R+\frac{1}{2}}) 
    \end{equation}
    uniformly with respect to $x \in \R$.
    
    \noindent $(b)$   If furthermore $\sup_{h\in \Hr, x, v\in \R}\left|\partial_x^{(2R)} f_{X_0|\Delta_h = v}(x) \right|<+\infty$ and $\lim_{x\to -\infty}\partial_x^{(2R)} f_{X_0|\Delta_h = v}(x)=0$ for every $v\!\in \R$, then 
  \begin{equation}
        F_{ X_h}(x) = F_{ X_0}(x) + \sum_{r=1}^{R} \frac{h^r}{r!} \esp{P_{r}( X_0)\mbox{\bf 1}_{\{ X_0\le x\}} } + \calig{O}(h^{R+\frac{1}{2}})
    \end{equation}
     uniformly with respect to $x \in \R$.
\end{Proposition}

\begin{proof}[Proof]
The case $R=0$ is trivial, using the expansion~\eqref{eq:taylor_nested} and the convention $\sum_{r=1}^0 = 0$.
Let $g:\R\to \R$ be an infinitely differentiable test function with compact support. We apply the expansion~\eqref{eq:taylor_nested} to the smooth function $g$  where coefficients $c(r,2r+1)$, $r=1,\ldots,R-1$ and $c(r, 2R)$, $r=R,\ldots,2R-1$ are given by~\eqref{eq:taylor_coeff} and the remainder term $\calig{R}^{g}_{2R+1}$ is given by~\eqref{eq:taylor_reste}. 

We first note that, for every $\ell\!\in \{1,\ldots, 2R\}$, 
\[
 \esp{g^{(\ell)}( X_0) b_{r, \ell-r}(Y)} = \int_{\R^2} g^{(\ell)}(x) b_{r, \ell-r}(y) f_{X_0,Y}(x,y)\d x \d y.
\]
Then, performing successively $\ell$ integrations by parts yields
\[
 \esp{g^{(\ell)}( X_0) b_{r, \ell-r}(Y)} =  \int_{\R} g(x)  \int_{\R} (-1)^{\ell} b_{r, \ell-r}(y) \partial_x^{(\ell)}f_{X_0,Y}(x,y) \d y \d x.
 \]
 As for the remainder term,
\begin{align*}
 \calig{R}^{g}_{2R+1} &= \frac{1}{(2R)!} \esp{\int_0^{X_h-X_0} g^{(2R+1)}(t+ X_0) ( X_h- X_0-t)^{2R} \d t}\\
 &= \frac{1}{(2R)!}  \esp{\int_0^1 g^{(2R+1)}(X_0 + s \Delta_h) (\Delta_h)^{2R+1} (1-s)^{2R} \d s}\\
 &= \frac{1}{(2R)!}  \int_0^1 \int_{\R^2} g^{(2R+1)}(x) v^{2R+1} f_{X_0, \Delta_h}(x - sv, v) \d x \d v (1-s)^{2R} \d s.
\end{align*}
Performing successively $2R+1$ integrations by parts yields
\begin{align*}
    \calig{R}^{g}_{2R+1} &= \frac{1}{(2R)!} \int_0^1 \int_{\R} \pa{\int_{\R} g(x) \partial_x^{(2R+1)} f_{X_0, \Delta_h}(x - sv, v) \d x} v^{2R+1} \d v (1-s)^{2R} \d s, \\
    &= \int_{\R} g(x) r(h,x) \d x, 
\end{align*}
where  
\begin{equation}\label{eq::rest_def}
    \begin{aligned}
    r(h, x) &= \frac{1}{(2R)!}  \int_{0}^1 \int_{\R} \partial_x^{(2R+1)} f_{X_0, \Delta_h}(x - sv, v) v^{2R+1} \d v (1-s)^{2R} \d s, \\
    &= \frac{1}{(2R)!}  \int_{0}^1 \int_{\R} \partial_x^{(2R+1)} f_{X_0|\Delta_h = v}(x - sv) f_{\Delta_h}(v) v^{2R+1} \d v (1-s)^{2R} \d s. \\
    \end{aligned}
\end{equation}
 Plugging these identities in~\eqref{eq:taylor_nested}, we get  that, for every test-function $g$,
 \[
 \esp{g( X_h)} = \int_{\R^d} g(x)\left[f_{ X_0}(x) + f_{ X_0}(x) \sum_{r=1}^{R} \frac{h^r}{r!} P_{r}(x) + \widetilde r(h,x)\right]\d x, 
 \]
 where 
 \[
  \widetilde r(h,x) = f_{ X_0}(x) \sum_{r=R+1}^{2R-1} \frac{h^r}{r!} P_{r}(x) + r(h,x),
 \]
 Hence 
 \begin{equation}
  \label{eq::first_density}
  f_{ X_h}(x) = f_{ X_0}(x) + f_{ X_0}(x) \sum_{r=1}^{R} \frac{h^r}{r!} P_{r}(x)    + \widetilde r(h,x).
 \end{equation}
 The continuity of the function on the right hand side of the above equality will establish the announced expansion, provided we show that $r(h,x)= \calig{O}\big(h^{R+\frac 12}\big)$ uniformly with respect to $x\!\in \R$.  It follows from the boundedness assumption made on $\partial_x^{(2R+1)} f_{X_0|\Delta_h = v}(x)$ that
\begin{equation}
\label{eq::rest_upper_bound}
    \abs{r(h, x)} \le \frac{1}{(2R)!}  \sup_{x,v \in \R} \Big|\partial_x^{(2R+1)} f_{X_0|\Delta_h = v}(x) \Big|\esp{\frac{\abs{\Delta_h}^{2R+1}}{2R+1}}\le\frac{C_{_{\Xi,R}}}{(2R+1)!}h^{R+\frac12}, 
\end{equation} 
owing to the upper-bound established in~\eqref{eq:Yh-Y_0:2R+1} for $\esp{| X_h- X_0|^{2R+1}}$, since $\Xi \in L^{2R+1}$.

\noindent $(b)$ The claim amounts to integrating Equation~\eqref{eq::first_density}, provided we show that the integrals of $P_{r}(x) f_{ X_0} (x)$, for $r=1, \ldots,2R-1$, are at least semi-convergent and that, for all $b\in\RR$,  $\int_{-\infty}^b r(h,x) \d x = \calig{O}(h^{R+\frac12})$.  Owing to Fubini's Theorem, using the definition~\eqref{eq::rest_def} of $r(h,x)$, we have for all $a < b \in \R$,  
\begin{multline*}
    \int_a^b r(h, x) \d x = \frac{1}{(2R)!}  \int_{0}^1 \int_{\R} \pa{\partial_x^{(2R)} f_{X_0|\Delta_h = v}(b - sv) - \partial_x^{(2R)} f_{X_0|\Delta_h = v}(a-sv)} \\ \times f_{\Delta_h}(v) v^{2R+1} \d v (1-s)^{2R} \d s
\end{multline*}
The assumption $\sup_{h\in \Hr,x,v\in \R}| f^{(2R)}_{ X_0| \Delta_h=v}(x)|<+\infty$ and the upper bound~\eqref{eq:Yh-Y_0:2R+1}, yield that $\int_a^b \abs{r(h, x)} \d x < +\infty$. 
Hence, owing to Lebesgue's Dominated Convergence Theorem and to the assumption \newline $\lim_{x\to -\infty}f^{(2R)}_{ X_0| \Delta_h=v}(x)=0$, we get
\begin{equation}\label{eq::finite_rest}
\begin{aligned}
    \int_{-\infty}^b r(h,x) \d x &= \frac{1}{(2R)!}  \int_{0}^1 \int_{\R} \partial_x^{(2R)} f_{X_0|\Delta_h = v}(b - sv) f_{\Delta_h}(v) v^{2R+1} \d v (1-s)^{2R} \d s, \\
    &= \frac{1}{(2R)!} \int_0^1 \esp{\partial_x^{(2R)} f_{X_0|\Delta_h}(b - s \Delta_h) (\Delta_h)^{2R+1}} (1-s)^{2R} \d s. 
\end{aligned} 
\end{equation}
and then, likewise~\eqref{eq::rest_upper_bound}, using the boundedness of $f^{(2R)}_{_{ X_0\,|\, \Delta_h=v}}(x) $,
\begin{equation}\label{eq::rest_semiconvergent}
   \int_{-\infty}^b r(h,x) \d x  = \calig{O}(h^{R+\frac12}).       
\end{equation}
Owing to Equation~\eqref{eq::first_density}, if we take $h_1,\ldots, h_{2R-1} \in\H$ pairwise distinct,  we get, for all $i=1,\ldots, 2R-1$, 
\begin{equation}\label{eq::p_vandermond_system}
 \sum_{r=1}^{2R-1} h_i^{r-1} \left( \frac{ P_r(x)}{r!} f_{ X_0} (x) \right) = \rho_i(x), 
\end{equation}
with 
\begin{equation}\label{eq::Xi_def}
\rho_i(x) = \frac{f_{ X_{h_i}}(x) - f_{ X_{0}}(x) - r(h_i, x)}{h_i}, \quad i=1,\ldots,2R-1.
\end{equation}
Hence we get a Vandermonde system, $V u(x) = \rho(x)$ with 
$$V = V(h_1,\ldots, h_{2R-1}) = \left(
\begin{array}{ccccc}
 1 & h_1 & h_1^2 & \ldots & h_1^{2R-2} \\
 1 & h_2 & h_2^2 & \ldots & h_2^{2R-2} \\
 \vdots & \vdots & \vdots & \ldots & \vdots \\
 1 & h_{2R-1} & h_{2R-1}^2 & \ldots & h_{2R-1}^{2R-2}
\end{array} \right),
$$
$$u(x)= (u_1(x), \ldots,u_{2R-1}(x)) \quad \mbox{with} \quad u_r(x) =  \frac{ P_r(x)}{r!} f_{ X_0} (x)$$
and
$$\rho(x) = (\rho_1(x),\ldots,\rho_{2R-1}(x)).$$
We set, for all $j=1, \ldots,2R-1$, 
$$\widetilde V_j(x) = \widetilde V_j(h_1,\ldots,h_{2R-1}, \rho(x)) = \left(
\begin{array}{ccccccc}
 1 	& \ldots & h_1^{j-2} 	& \rho_1(x) 	& h_1^j 	& \ldots & h_1^{2R-2} \\
 1 	& \ldots & h_2^{j-2} 	& \rho_2(x) 	& h_2^j 	& \ldots & h_2^{2R-2} \\
 \vdots & \ldots & \vdots 	& \vdots 	& \vdots 	& \ldots & \vdots \\
 1 	& \ldots & h_{2R-1}^{j-2} & \rho_{2R-1}(x) 	& h_{2R-1}^j 	& \ldots & h_{2R-1}^{2R-2} 
\end{array} \right).$$
By expanding along the $j^{th}$ column, the determinant of $\widetilde V_j$ writes
$$\det(\widetilde V_j(x)) = \sum_{i=1}^{2R-1} (-1)^{i+j} d_{ij} \rho_i(x), $$
where $d_{ij} := d_{ij}(h_1,\ldots,h_{i-1},h_{i+1},\ldots,h_{2R})\in\RR$ with
$$d_{ij}(h_1,\ldots,h_{i-1},h_{i+1},\ldots,h_{2R-1}) = \det\left(
\begin{array}{cccccc}
 1 	& \ldots & h_1^{j-2} 	&  h_1^j 	& \ldots & h_1^{2R-2} \\
 \vdots & \ldots & \vdots 	&  \vdots 	& \ldots & \vdots \\
 1 	& \ldots & h_{i-1}^{j-2}&  h_{i-1}^j 	& \ldots & h_{i-1}^{2R-2} \\
 1 	& \ldots & h_{i+1}^{j-2}&  h_{i+1}^j 	& \ldots & h_{i+1}^{2R-2} \\
 \vdots & \ldots & \vdots 	&  \vdots 	& \ldots & \vdots \\
 1 	& \ldots & h_{2R-1}^{j-2} &  h_{2R-1}^j 	& \ldots & h_{2R-1}^{2R-2} 
\end{array} \right).$$
Hence, owing to Cramer's rule, the solution of the Vandermonde system writes
$$u_j(x) = \frac{\det\left( \widetilde V_j(x)\right)}{\det(V)} = \frac{1}{\det(V)}\sum_{i=1}^{2R-1}  (-1)^{i+j} d_{ij} \rho_i(x).$$
Finally, since we saw that for all $i = 1, \ldots, 2R-1$, the integral of $\rho_i(x)$ is semi-convergent, we deduce the semi-convergence of the integral   
$$\int_{-\infty}^b  P_r(x) f_{ X_0} (x) \d x  =  \frac{r!}{\det(V)}\sum_{i=1}^{2R-1} (-1)^{i+r} d_{ir} \int_{-\infty}^b \rho_i(x) \d x,$$
where, owing to the expression~\eqref{eq::Xi_def} and to~\eqref{eq::rest_semiconvergent},  $\int_{-\infty}^b \rho_i(x) \d x$ is finite,
which concludes the proof.

\end{proof}

\subsection{Strong convergence rate}
We conclude by showing the strong convergence rate~\eqref{strong_error} for the nested Monte Carlo estimator.
The following Lemma  is more or less standard (see for instance~\cite{Av09}).
\begin{Lemma}\label{lem:StrEr1}
 Let $\xi$ and $\xi'$ be two  real valued random variables lying in $L^p$, $p\ge 1$, with densities $f_{\xi}$ and $f_{\xi'}$ respectively. Then, for every $x \!\in \R$,  
\begin{equation}\label{eq:StrErQuantile}
\Big\|   \mbox{\bf 1}_{\{\xi\le x \}}-\mbox{\bf 1}_{\{\xi' \le x\}} \Big\|_2^2\le \Big(p^{\frac{p}{p+1}} +p^{\frac{1}{p+1}}\Big)\left(\big\|f_{\xi}\big\|_{\sup} +\big\|f_{\xi'}\big\|_{\sup}\right)^{\frac{p}{p+1}}\big\|\xi-\xi'\big\|_p^{\frac{p}{p+1}}.
\end{equation}
\end{Lemma}

\begin{proof}[Proof] Let $L>0$. Note that
\begin{eqnarray*}
\Big\|   \mbox{\bf 1}_{\{\xi\le x \}}-\mbox{\bf 1}_{\{\xi'\le x \}} \Big\|_2^2 &=& \PP\big(\xi\le x\le \xi'\big)+ \PP\big(\xi'\le x\le \xi\big)\\
&\le& \PP\big(\xi \le x,\, \xi'\ge x+L\big)+\PP\big(\xi \le x \le \xi'\le x+L\big)\\
&& + \PP\big(\xi'\le x,  \xi\ge x+L\big)+\PP\big(\xi' \le x \le \xi\le x+L\big)\\
&\le &  \PP\big(\xi'-\xi\ge L) + \PP\big(\xi-\xi'\ge L)\\ 
&& + \PP(\xi'\!\in [x, x+L]) +  \PP(\xi\!\in [x, x+L]) \\
&=& \PP\big(|\xi'-\xi|\ge L) + \PP(\xi\!\in [x, x+L]) +  \PP(\xi'\!\in [x, x+L]) \\
&\le& \frac{\esp{|\xi'-\xi|^p}}{L^p} +L \Big( \big\|f_{\xi}\big\|_{\sup}+ \big\|f_{\xi'}\big\|_{\sup}   \Big).
\end{eqnarray*}

A straightforward optimization in $L$ yields the announced result.
\end{proof} 

The strong convergence is a consequence of the Proposition~\ref{prop:nested_weakII} combined with the previous Lemma and Lemma~\ref{lemma::strong_conv_useful}.

\begin{Proposition} 
\label{prop::strong_error_indicator}
Assume $X\!\in L^p$, $p\ge 2$. Under the assumptions of Proposition~\ref{prop:nested_weakII} $(a)$ with $R=0$  and if  the density $f_{ X_0}$ is bounded, then there exists  $h_0= \frac{1}{K_0}\!\in \Hr\setminus\{0\}$ such that, for every $h, h'\!\in (0, h_0]$
\[
\Big\|   \mbox{\bf 1}_{\{ X_h\le x \}}-\mbox{\bf 1}_{\{ X_{h'}\le x \}} \Big\|_2^2\le C |h-h'|^{\frac{p}{2(p+1)} },
\]
where $C =     2^{\frac{p}{p+1}}  \Big(p^{\frac{p}{p+1}} +p^{\frac{1}{p+1}}\Big) (\big\|f_{_{ X_0}}\big\|_{\sup} +1)^{\frac{p}{p+1}} \big (4\, B_p \big\|X\big\|_p \big)^{\frac{p}{p+1}}$.  This means that the strong approximation error assumption holds with $\beta =\frac{p}{2(p+1)}\!\in (0,\frac 12)$.
\end{Proposition}
\begin{proof}[Proof] It follows from Proposition~\ref{prop:nested_weakII} $(a)$ that 
\[
f_{ X_h}(x) \le f_{ X_0}(x) + o(h^{\frac 12})\quad \mbox{ uniformly with respect to $x\!\in \R$}.
\]
Consequently, there exists an $h_0= \frac{1}{K_0}  \!\in \Hr\setminus\{0\}$ such that, for every $h\!\in (0, h_0]$,  
\[
\forall\, x\!\in \R, \; f_{ X_h}(x) \le f_{ X_0} (x)+1.
\]
Plugging the above  bound and~\eqref{eq:Yh-Yhprime} in Inequality~\eqref{eq:StrErQuantile} of Lemma~\ref{lem:StrEr1} applied with $\xi= X_h$ and $\xi'=  X_{h'}$ completes the proof.  
\end{proof} 
\printbibliography

\end{document}